\documentclass[12pt]{amsart}

\usepackage{amsmath}
\usepackage{amsfonts}
\usepackage{amsthm}

\usepackage[cmtip,arrow]{xy}
\usepackage{pb-diagram, pb-xy}

\newtheorem{theorem}{Theorem}

\newtheorem{lemma}{Lemma}
\newtheorem{proposition}{Proposition}
\newtheorem{definition}{Definition}

\newcommand{\C}{\mathcal{C}}

\newcommand{\F}{\mathcal{F}}

\renewcommand{\L}{\mathcal{L}}

\begin{document}
\title{Gorenstein injective envelopes and covers over two sided noetherian rings}
\author{Alina Iacob}
\thanks{key words: Gorenstein injective module, Gorenstein injective envelope,  Gorenstein injective cover, strongly cotorsion module}

\maketitle
\markright{Gorenstein injective envelopes and covers over two sided noetherian rings}

\begin{abstract}
We prove that the class of Gorenstein injective modules is both enveloping and covering over a two sided noetherian ring such that the character modules of Gorenstein injective modules are Gorenstein flat. 
In the second part of the paper we consider the connection between the Gorenstein injective modules and the strongly cotorsion modules. We prove that when the ring $R$ is commutative noetherian of finite Krull dimension, the class of Gorenstein injective modules coincides with that of strongly cotorsion modules if and only if the ring $R$ is in fact Gorenstein.
\end{abstract}

\section{Introduction}

The class of Gorenstein injective modules was introduced by Enochs and Jenda in 1995 (\cite{enochs:95:gorenstein}). Together with the Gorenstein projective and with the Gorenstein flat modules, they are the key elements of Gorenstein homological algebra. So it is natural to consider the question of the existence of the Gorenstein injective envelopes and covers. Their existence is known over Gorenstein rings (\cite{enochs:00:relative}). It is also known (\cite{enochs:12:gorenstein.injective.covers}) that when $R$ is a commutative noetherian ring such that  the character modules of Gorenstein injective modules are Gorenstein flat, then the class of Gorenstein injective modules is both enveloping and covering.
We extend these results to two sided noetherian rings (not necessarily commutative). More precisely, we show that when $R$ is a two sided noetherian ring and the character module of any Gorenstein injective left $R$-module is Gorenstein flat, the class of Gorenstein injective left $R$-modules is both enveloping and covering.\\
 Then we consider the relation between the Gorenstein injective modules and the strongly cotorsion modules. It is known that over a commutative noetherian ring $R$ of finite Krull dimension, any Gorenstein injective module is strongly cotorsion (\cite{salarian:03: strongly cotorsion}). We prove that over such a ring the class of Gorenstein injective modules coincides with that of strongly cotorsion modules if and only if the ring $R$ is in fact Gorenstein.

\section{Preliminaries}

Throughout this section $R$ denotes an associative ring with unity. By $R$-module we mean a left $R$-module.\\

\begin{definition} (\cite{enochs:00:relative}, definition 10.1.1)
A module $G$ is Gorenstein injective if there is an exact complex $ \ldots \rightarrow E_1 \rightarrow E_0
\rightarrow E_{-1} \rightarrow \ldots $ of
injective modules which remains exact under application of the functors $Hom (Inj, -)$ , where $Inj$ stands for all injective modules, and such that $G = Ker(E_0 \rightarrow E_{-1})$.
\end{definition}

We will use the notation $\mathcal{GI}$ for the class of Gorenstein injective modules.\\

A stronger notion is that of strongly Gorenstein injective module:\\
\begin{definition} (\cite{bennis:07:stronglygorenstein}, definition 2.1)
An $R$-module $M$ is strongly Gorenstein injective if there exists an exact and $Hom(Inj,-)$ exact sequence $$\ldots \rightarrow E \xrightarrow{f} E \xrightarrow{f} E \rightarrow \ldots$$ with $E$ injective and with $M = Ker(f)$.
\end{definition}

It is known (\cite{bennis:07:stronglygorenstein}, Theorem on page 3) that a module is Gorenstein injective if and only if it is a direct summand of a strongly Gorenstein injective one.

The Gorenstein flat modules are defined in terms of the tensor product.\\
\begin{definition} (\cite{enochs:00:relative}, Definition 10.3.1)
An $R$-module $G$ is Gorenstein flat if there exists an exact and $Inj \otimes -$ exact sequence of flat modules $$\ldots \rightarrow F_1 \rightarrow F_0 \rightarrow F_{-1} \rightarrow \ldots $$ such that $G=Ker(F_0 \rightarrow F_{-1})$.
\end{definition}

We will use the notation $\mathcal{GF}$ for the class of Gorenstein flat modules.

We recall that the character module of a left $R$-module $M$ is the right $R$-module $M^+ = Hom_Z (M, Q/Z)$.\\
It is known (\cite{holm:04:gorenstein}, Theorem 3.6) that if the ring $R$ is right coherent then a module $G$ is Gorenstein flat if and only if its character module $G^+$ is Gorenstein injective.

Given a class of $R$-modules $\mathcal{F}$, we will denote as usual by $\F^\bot$ the class of all $R$-modules $M$ such that $Ext^1(F,M)=0$ for every $F \in \F$.\\
The left orthogonal class of $\F$, denoted $^\bot \F$, is the class of all $_RN$ such that $Ext^1(N,F)=0$ for every $F \in \F$.\\

We recall that a pair $(\mathcal{L}, \mathcal{C})$ is a \emph{cotorsion pair} if $\mathcal{L} ^\bot = \mathcal{C}$ and $^\bot \mathcal{C} = \mathcal{L}$.\\
A cotorsion pair $(\mathcal{L}, \mathcal{C})$ is \emph{complete} if for every $_RM$ there exists exact sequences $ 0 \rightarrow C \rightarrow L \rightarrow M \rightarrow 0$ and $0 \rightarrow M \rightarrow C'\rightarrow L' \rightarrow 0$ with $C$, $C'$ in $\mathcal{C}$ and $L$, $L'$ in $\mathcal{L}$.\\

\begin{definition} (\cite{garcia:99:covers}, Definition 1.2.10)
 A cotorsion pair $(\L, \C)$ is called hereditary
if one of the following equivalent statements hold:
\begin{enumerate}
\item $\L$ is resolving, that is, $\L$ is closed under taking kernels of epimorphisms.
\item $\C$ is coresolving, that is, $\C$ is closed under taking cokernels of monomorphisms.
\item $Ext^i (F, C) = 0$ for any $F \in \F$ and $C\in \C$ and $i\geq 1$.
\end{enumerate}
\end{definition}

Our main results in section 3 are about the existence of covers and envelopes with respect to the class of Gorenstein injective modules. We recall the following definitions:\\

\begin{definition} (\cite{garcia:99:covers}, Definition 1.2.3) Let $\mathcal{C}$ be a class of left $R$-modules. A homomorphism $\phi: M \rightarrow C$ is a $\mathcal{C}$-preenvelope of $M$ if $C \in \mathcal{C}$ and if for any homomorphism $\phi ':M \rightarrow C'$ with $C' \in \mathcal{C}$, there exists $u \in Hom_R(C,C')$ such that $\phi ' = u \phi$.\\
A $\mathcal{C}$-preenvelope $\phi: M \rightarrow C$ is a $\mathcal{C}$-envelope if any endomorphism $u \in Hom_R(C,C)$ such that $\phi = u \phi$ is an automorphism of $C$.\\
(Pre)covers and covers are defined dually.\\
In the case when $\C$ is the class of Gorenstein injective modules, a $\C$-(pre)cover (preenvelope respectively) is called a Gorenstein injective (pre)cover (preenvelope respectively).

\end{definition}


Duality pairs were introduced by Holm and J{\o}rgensen in \cite{holm:09:duality}. We recall their definition (the opposite ring is denoted $R^{op}$):\\
\begin{definition} (\cite{holm:09:duality}, Definition 2.1)
A duality pair over a ring $R$ is a pair $(\textsc{M},\textsc{C})$ where $\textsc{M}$ is a class of $R$-modules and $\textsc{C}$ is a class of $R^{op}$-modules, subject to the following conditions:\\
(1) For an $R$-module $M$, one has $M \in \textsc{M}$ if and only if $M^+ \in \textsc{C}$.\\
(2) $\textsc{C}$ is closed under direct summands and finite direct sums.
\end{definition}

A duality pair $(\textsc{M},\textsc{C})$ is called (co)product closed if the class $\textsc{M}$ is closed under (co)products in the category of all $R$-modules. \\
A duality pair $(\textsc{M},\textsc{C})$ is called \emph{perfect } if it is coproduct-closed, if $\textsc{M}$ is closed under extensions, and if $R$ belongs to $\textsc{M}$.

We also recall that a cotorsion pair $(\mathcal{L}, \mathcal{C})$  is \emph{perfect} if $\mathcal{C}$ is an enveloping class and $\mathcal{L}$ is a covering class.\\

The following result is \cite{holm:09:duality}, Theorem 3.1:\\
\textbf{Theorem}. Let $(\textsc{M},\textsc{C})$ be a duality pair. Then $\textsc{M}$ is closed under pure submodules, pure quotients, and pure extensions. Furthermore, the following hold:\\
(a) If $(\textsc{M},\textsc{C})$ is product-closed then $\textsc{M}$ is preenveloping.\\
(b) If $(\textsc{M},\textsc{C})$ is coproduct-closed then $\textsc{M}$ is covering.\\
(c) If $(\textsc{M},\textsc{C})$ is perfect then $(\textsc{M}, \textsc{M}^\bot)$ is a perfect cotorsion pair.


\section{Gorenstein injective envelopes and covers}

We recall the following results from \cite{enochs:12:gorenstein.injective.covers}:\\

\begin{lemma} (\cite{enochs:12:gorenstein.injective.covers}, Corollary 1)
If $R$ is left noetherian then $(^\bot \mathcal{GI}, \mathcal{GI})$ is a complete hereditary cotorsion pair.
\end{lemma}

\begin{theorem} (\cite{enochs:12:gorenstein.injective.covers}, Proposition 2)
Let $R$ be a left noetherian ring. The class of Gorenstein injective modules is enveloping if and only if the class $^\bot \mathcal{GI}$ is covering.
\end{theorem}

We prove that if the ring $R$ is two sided noetherian and has the property that the character modules of Gorenstein injective left $R$-modules are Gorenstein flat right $R$-modules, then the class of Gorenstein injective modules is enveloping.
Since the class $^\bot \mathcal{GI}$ is closed under arbitrary direct sums, it suffices to prove that $^\bot \mathcal{GI}$ is the left half of a duality pair. Then by \cite{holm:09:duality}, Theorem 3.1 (b), the class $^\bot \mathcal{GI}$ is covering. By Theorem 1, $\mathcal{GI}$ is enveloping in this case.


We start with the following \\

\begin{lemma}
Let $R$ be a two sided noetharin ring such that for any (left) Gorenstein injective module $M$, its character module, $M^+$ is a (right) Gorenstein flat module. Then a left $R$-module $K$ is in the class $^\bot \mathcal{GI}$ if and only if its character module $K^+$ is in $\mathcal{GF}^\bot$.
\end{lemma}

\begin{proof}
$\Rightarrow$ Let $K \in ^\bot \mathcal{GI}$. For any Gorenstein flat right $R$-module $B$ we have $B^+ \in \mathcal{GI}$ (\cite{holm:04:gorenstein}, Theorem 3.6). It follows that $Ext^1(K, B^+)=0$. Then $Ext^1 (B, K^+) \simeq Ext^1 (K,B^+)=0$. So $K^+ \in \mathcal{GF}^ \bot$.\\
$\Leftarrow$ Assume that $K$ is a left $R$-module such that $K^+ \in \mathcal{GF}^ \bot$.\\
Since $R$ is noetherian there exists an exact sequence $0 \rightarrow K \rightarrow G \rightarrow V \rightarrow 0$ with $G$ Gorenstein injective and with $V \in ^ \bot \mathcal{GI}$ (by Lemma 1). This gives an exact sequence $0 \rightarrow V^+ \rightarrow G^+ \rightarrow K^+ \rightarrow 0$ with $G^+$ Gorenstein flat. Since $V \in ^\bot \mathcal{GI}$ we have that $V^+ \in \mathcal{GF}^ \bot$. Since $K^+$ is also in $\mathcal{GF}^ \bot$ it follows that $G^+ \in \mathcal{GF} \cap \mathcal{GF}^ \bot$. Thus $G^+$ is flat and therefore $G$ is injective (by \cite{enochs:00:relative}, Theorem 3.2.16).\\
So we have an exact sequence $0 \rightarrow K \rightarrow G \rightarrow V \rightarrow 0$ with $G \in Inj \subseteq ^\bot \mathcal{GI}$ and with $V \in ^\bot \mathcal{GI}$. Since $(^\bot \mathcal{GI}, \mathcal{GI})$ is a hereditary cotorsion pair, it follows that $K \in ^\bot \mathcal{GI}$.
\end{proof}

Now we can prove:\\
\begin{theorem}
Let $R$ be a two sided noetherian ring such that the character modules of Gorenstein injective left $R$-modules are Gorenstein flat right $R$-modules. Then $(^\bot \mathcal{GI}, \mathcal{GF}^\bot)$ is a duality pair. In particular the class $^\bot \mathcal{GI}$ is covering.
\end{theorem}

\begin{proof}
By Lemma 2 we have that $K \in ^\bot \mathcal{GI}$ if and only if $K^+ \in \mathcal{GF}^\bot$.\\
Any right orthogonal class (in particular $\mathcal{GF}^\bot$) is closed under direct products, and so it is closed under finite direct sums.
Also, any right orthogonal class (so $\mathcal{GF}^\bot$ in particular) is closed under direct summands.\\
Thus $(^\bot \mathcal{GI}, \mathcal{GF}^\bot)$  is a duality pair. Since the class $^\bot \mathcal{GI}$ is closed under direct sums, it follows (by \cite{holm:09:duality}, Theorem 3.1) that $^\bot \mathcal{GI}$ is covering.
\end{proof}

\begin{theorem}
Let $R$ be a two sided noetherian ring such that the character modules of Gorenstein injective left $R$-modules are Gorenstein flat right $R$-modules. Then the class of Gorenstein injective modules is enveloping.
\end{theorem}

\begin{proof}
This follows from Theorem 1 and Theorem 2.
\end{proof}

We prove that over the same kind of rings the class of Gorenstein injective modules is also covering.\\

\begin{theorem}
Let $R$ be a two sided noetherian ring such that the character modules of Gorenstein injective modules are Gorenstein flat. Then $(\mathcal{GI}, \mathcal{GF})$ is a duality pair.
\end{theorem}

\begin{proof}
- We prove first that $K \in \mathcal{GI}$ if and only if $K^+ \in \mathcal{GF}$.\\
One implication is a hypothesis we made on the ring.\\
Assume that $K^+ \in \mathcal{GF}$. Since the ring $R$ is left noetherian there exists an exact sequence $0 \rightarrow K \rightarrow G \rightarrow L \rightarrow 0$ with $G$ Gorenstein injective and $L \in ^\bot \mathcal{GI}$. Therefore we have an exact sequence $0 \rightarrow L^+ \rightarrow G^+ \rightarrow K^+ \rightarrow 0$ with $G^+$ Gorenstein flat, and $L^+ \in \mathcal{GF}^\bot$ (by Lemma 2). Then $Ext^1(K^+, L^+)=0$, so $G^+ \simeq L^+ \oplus K^+$, and therefore $L^+$ is a Gorenstein flat right $R$-module. It follows that $L^+ \in \mathcal{GF} \bigcap \mathcal{GF}^ \bot$, so $L^+$ is a flat right $R$-module, and therefore $L$ is injective (by \cite{enochs:00:relative}, Theorem 3.2.16). The exact sequence $0 \rightarrow K \rightarrow G \rightarrow L \rightarrow 0$ with $G$ Gorenstein injective and with $L$ injective, gives that $K$ has finite Gorenstein injective dimension. By \cite{christensen:06:gorenstein} Lemma 2.18, there exists an exact sequence $0 \rightarrow B \rightarrow H \rightarrow K \rightarrow 0$ with $B$ Gorenstein injective and with $id_R H = Gid_R K < \infty$. This gives an exact sequence $0 \rightarrow K^+ \rightarrow H^+ \rightarrow B^+ \rightarrow 0$. Both $B^+$ and $K^+$ are Gorenstein flat modules, so $H^+$ is also Gorenstein flat. Since $id_R H$ is finite it follows that $H^+$ has finite flat dimension (\cite{enochs:00:relative}, Theorem 3.2.19). But a Gorenstein flat module of finite flat dimension is flat (\cite{enochs:00:relative}, Corollary 10.3.4). So $H^+$ is flat and therefore $H$ is injective. Thus $Gid_R K =0$.\\

- We can prove now that $(\mathcal{GI}, \mathcal{GF})$ is a duality pair.\\
The first condition from the definition holds by the above.\\
It is known that the class of Gorenstein flat modules is closed under (finite) direct sums, and under direct summands (\cite{christensen:11:beyond}, Theorem 4.14).
\end{proof}


In the proof of the result about the existence of the Gorenstein injective envelopes we will use the following lemma:\\

\begin{lemma}
A direct sum of modules is a pure submodule of the direct product of the modules.
\end{lemma}
\begin{proof}
Let $(X_i)_{i \in I}$ be a family of $R$-modules. The direct sum $\oplus_{i \in I} X_i$ of the family is the direct limit of the finite direct sums $Y_j= \oplus_{i \le j} X_i$ of these modules. Any finite direct sum $Y_j$ is a direct summand of the direct product $Y = \prod X_i$, so it is a pure submodule of $Y$. Therefore the sequence $ 0 \rightarrow A \otimes Y_j \rightarrow A \otimes Y $ is exact, for any $R$-module $A$. Then $0 \rightarrow \underrightarrow{lim} (A \otimes Y_j) \rightarrow \underrightarrow{lim} (A \otimes Y)$ is exact.
Since the tensor product commutes with direct limits we have an exact sequence $0 \rightarrow A \otimes (\underrightarrow{lim} Y_j) \rightarrow A \otimes Y$. So $0 \rightarrow A \otimes (\oplus_{i \in I} X_i) \rightarrow A \otimes \prod X_i $ is exact for any $R$-module $A$.
\end{proof}

\begin{theorem}
Let $R$ be a two sided noetherian ring such that the character modules of Gorenstein injective left $R$-modules are Gorenstein flat right $R$-modules. Then the class of Gorenstein injective modules is covering in $R-Mod$.
\end{theorem}

\begin{proof}
By Theorem 4, $(\mathcal{GI}, \mathcal{GF})$ is a duality pair. Then by \cite{holm:09:duality} Theorem 3.1, the class of Gorenstein injective modules is closed under pure submodules. Since $\mathcal{GI}$ is closed under direct products and every direct sum is a pure submodule of a direct product (Lemma 3), it follows that $\mathcal{GI}$ is closed under arbitrary direct sums. Another application of \cite{holm:09:duality}, Theorem 3.1 gives that $\mathcal{GI}$ is covering.
\end{proof}



The following is an equivalent description of the condition that $(\mathcal{GI}, \mathcal{GF})$ is a duality pair.\\

\begin{proposition}
Let $R$ be a two sided noetherian ring such the character modules of Gorenstein injective modules are Gorenstein flat. The following are equivalent:\\
1. $(\mathcal{GI}, \mathcal{GF})$ is a duality pair;\\
2. The class of Gorenstein injective modules, $\mathcal{GI}$, is closed under pure submodules.
\end{proposition}

\begin{proof}
1. $\Rightarrow$ 2. follows from \cite{holm:09:duality}, Theorem  3.1.\\
2. $\Rightarrow$ 1. Assume that $K^+$ is a right Gorenstein flat module.\\
The ring $R$ is noetherian so there is an exact sequence $0 \rightarrow K \rightarrow G \rightarrow L \rightarrow 0$ with $G$ Gorenstein injective and with $L \in ^\bot \mathcal{GI}$. 
By the proof of Theorem 4, $G^+ \simeq L^+ \oplus K^+$, so the sequence $ 0 \rightarrow K \rightarrow G \rightarrow L \rightarrow 0$ is pure exact. Since $\mathcal{GI}$ is closed under pure submodules it follows that $K$ is Gorenstein injective. Thus $K$ is Gorenstein injective if and only if its character module, $K^+$, is Gorenstein flat.\\
Since the class of Gorenstein flat modules is closed under direct sums and under direct summands, it follows that $(\mathcal{GI}, \mathcal{GF})$ is a duality pair.
\end{proof}

Our results hold over two sided noetherian rings with the property that the character modules of left Gorenstein injective modules are Gorenstein flat right $R$-modules. It is known that this class of rings includes the Gorenstein rings (\cite{enochs:00:relative}), as well as the commutative noetherian rings with dualizing complexes (\cite{holm:09:duality}). But these are not the only rings with this property. We proved (\cite{iacob:14:gorenstein.flat.preenvelops}) that every two sided noetherian ring $R$ with $i.d._{R^{op}} R < \infty$ has the desired property. 

\begin{theorem}(\cite{iacob:14:gorenstein.flat.preenvelops}, Theorem 2)
Let $R$ be a two sided noetherian ring such that $i.d._{R^{op}} R \le n$ for some positive integer n. Then the character modules of Gorenstein injective left $R$-modules are Gorenstein flat right $R$-modules.
\end{theorem}

Then by Theorem 3, Theorem 5 and Theorem 6 we obtain:\\

\begin{theorem}
Let $R$ be a two sided noetherian ring such that $i.d._{R^{op}} R < \infty$. Then the class of left Gorenstein injective $R$-modules is both enveloping and covering in $R-Mod$.
\end{theorem}

\section{When is every strongly cotorsion module Gorenstein injective?}


We recall first the following:\\

\begin{definition}(\cite{enochs:12:injective envelopes}, Definition 2.6) 
 A module $_RM$ is called strongly cotorsion if $Ext^1(F,M) = 0$ for every module $F$ of finite flat dimension.
\end{definition}

It is known (see for example \cite{yan:10: strongly cotorsion}, Lemma 2.2) that $M$ is a strongly cotorsion module if and only if $Ext^i(F,M) = 0$ for all $i \ge 1$, for any module $F$ of finite flat dimension.

We will use the notation $\mathcal{SC}$ for the class of strongly cotorsion modules, and $\F$ for the class of modules of finite flat dimension.

We recall (\cite{enochs:05:dualizing.modules}) that a ring $R$ is said to be left $n$-perfect if every left $R$-module of finite flat dimension has projective dimension less than or equal to $n$.\\
Left perfect rings, commutative noetherian rings of finite Krull dimension, the universal enveloping algebra $\mathcal{U}(g)$ of a Lie algebra of dimension $n$, and $n$-Gorenstein rings are all examples of left $n$-perfect rings (by \cite{enochs:05:dualizing.modules}). 

\begin{proposition}
 If $R$ is left $n$-perfect then every Gorenstein injective left $R$-module is strongly cotorsion.
\end{proposition}

\begin{proof}
Let $M$ be a strongly Gorenstein injective left $R$-module. Then there is an exact sequence $0 \rightarrow M \rightarrow E \rightarrow M \rightarrow 0$ with $_RE$ an injective module. It follows that $Ext^i(-,M) \simeq Ext^1(-,M)$ for all $i \ge 1$.\\
Let $_RX$ be a module of finite flat dimension. Since $R$ is left $n$-perfect, the projective dimension of $X$ is less than or equal to $n$. So $Ext^i(X,M) = 0$ for all $i > n$. Then by the above we have that $Ext^i(X,M) = 0$ for all $i \ge 1$. Thus $M \in \mathcal{SC}$.\\
If $M'$ is a Gorenstein injective $R$-module then there exists a strongly Gorenstein injective module $M$ such that $M \simeq M' \oplus M"$ (by \cite{bennis:07:stronglygorenstein}). It follows that $Ext^1(X,M')=0$ for any $X \in \F$. So $M'$ is strongly cotorsion.
\end{proof}

\begin{proposition}
Let $R$ be a two sided noetherian ring such that $R$ is left $n$-perfect. If the class of strongly cotorsion modules coincides with that of Gorenstein injective modules then $i.d._{R^{op}} R < \infty$. 
\end{proposition}

\begin{proof}
Since $R$ is left noetherian, $(^\bot \mathcal{GI}, \mathcal{GI})$ is a complete hereditary cotorsion pair. By \cite{yan:10: strongly cotorsion}, Proposition 2.14, $(\mathcal{F}, \mathcal{SC})$ is also a cotorsion pair. \\ 
By our assumption, $\mathcal{SC}$ is the class of Gorenstein injective modules. It follows that $^\bot \mathcal{SC} = ^ \bot \mathcal{GI}$, that is, $\mathcal{F} = ^\bot \mathcal{GI}$. Since $Inj \subseteq ^\bot \mathcal{GI}$, it follows that every injective module has finite flat dimension.\\
Since $R$ is right noetherian we have that $i.d._{R^{op}} R = sup \{f.d. I | _RI \in Inj \}$ (by \cite{ywanaga:80:rings.self.injective}). It follows that $ i.d._{R^{op}} R  \le n$.
\end{proof}

In particular, any commutative noetherian ring of Krull dimension $n < \infty$ is left $n$-perfect. Over such a ring the condition $i.d._{R^{op}} R < \infty$ is equivalent with $R$ being a Gorenstein ring. So using Proposition 3 we obtain the following result:\\

\begin{theorem}
Let $R$ be a commutative noetherian ring of finite Krull dimension. 
The following are equivalent:\\
1) The class of Gorenstein injective modules coincides with that of strongly cotorsion modules.\\
2) $R$ is a Gorenstein ring.
\end{theorem}

\begin{proof}

1) $\Rightarrow$ 2) 
By Proposition 3, $i.d._{R^{op}} R  \le n$. Since $R$ is commutative noetherian and such that $i.d._{R^{op}} R  \le n$ it follows that $R$ is a Gorenstein ring.\\

 2)$\Rightarrow$ 1)  Since $R$ is a Gorenstein ring the class $^\bot \mathcal{GI}$ consists of all modules of finite projective dimension ([7]). By [7] again, this is the class $\mathcal{F}$ of modules of finite flat dimension. Then $\mathcal{GI} = (^\bot \mathcal{GI})^\bot = \mathcal{F}^\bot = \mathcal{SC}$.\\

\end{proof}

\bibliographystyle{plain}

Alina Iacob\\
Georgia Southern University,
Statesboro, GA 30460-8093\\
Email:  aiacob@GeorgiaSouthern.edu

\end{document}